\documentclass[12pt]{amsart}
\usepackage{latexsym,amssymb,amsmath}
\usepackage{enumerate}

\newcommand{\R}{{\mathbb R}}

\newcommand{\Z}{{\mathbb Z}}

\newcommand{\Q}{{\mathbb Q}}

\newtheorem{theorem}{Theorem}
\newtheorem{definition}{Definition}

\newtheorem{remark}{Remark}

\begin{document}
\title{On the Rationality of the Spectrum}

\author[Bose]{Debashish Bose}
\address{Debashish Bose, Department of Mathematics Shiv Nadar University, India}
\author[Madan]{Shobha Madan}
\address{Shobha Madan, Department of Mathematics, I.I.T. Kanpur, India}

\email{debashish.bose@snu.edu.in, madan@iitk.ac.in}


\begin{abstract}

Let $\Omega \subset \R$ be a compact set with measure $1$. If there exists a subset  $\Lambda \subset \R$ such that the set of exponential functions $E_{\Lambda}:=\{e_\lambda(x) = e^{2\pi i \lambda x}|_\Omega :\lambda \in \Lambda\}$ is an orthonormal basis for $L^2(\Omega)$, then $\Lambda$ is called  a spectrum for the set $\Omega$. A set $\Omega$ is said to tile $\R$ if there exists a set $\mathcal T$ such that $\Omega + \mathcal T = \R$. A conjecture of Fuglede suggests that Spectra and Tiling  sets are related. Lagarias and Wang \cite {LW1} proved that Tiling sets are always periodic and are rational. That any spectrum is also a periodic set was proved in \cite {BM1}, \cite {IK}. In this paper, we give some partial results to support the rationality of the spectrum.
\end{abstract}
\maketitle

\section{Introduction}
In this paper we explore the rationality of the spectrum in $\R$. 

Let $\Omega \subset \R^d$ be a (compact) set with Lebesgue measure
$|\Omega| = 1$.

\begin{definition}
$\Omega$ is said to be a {\bf $spectral$ $set$} if there exists a
subset  $\Lambda \subset \R^d$ such that the set of exponential
functions $E_{\Lambda}:=\{e_\lambda(x) = e^{2\pi i \lambda .x}|_\Omega :\lambda \in \Lambda\}$ is an orthonormal basis for the Hilbert space $L^2(\Omega)$. The set
$\Lambda$ is said to be a {\bf $spectrum$} for $\Omega$, and the
pair $(\Omega,\Lambda)$ is called a {\bf $spectral$ $pair$}.
\end{definition}

It is easy to see that for a spectral set $\Omega$, the spectrum need not be unique, and conversely given a spectrum $\Lambda$, there may be many sets $\Omega$ such that $(\Omega, \Lambda )$ is a spectral pair.

Interest in the spectrum arose from a conjecture due to Fuglede relating spectral and tiling properties of  sets. $\Omega$ is said to tile $\R^d$ if there exists a subset $ \mathcal T \subset \R^d$ such that $\Omega + \mathcal T$ is a partition (a.e.) of $\R^d$. $(\Omega, \mathcal T)$ is called a tiling pair and $\mathcal T$ is called a tiling set.

\medskip
\noindent {\bf Fuglede's Conjecture}\cite {Fug}. A set $\Omega \subset \R^d$ is a spectral set if and only if $\Omega$ tiles $\R^d$ by translations.

\medskip
The conjecture suggests that there could be a strong relationship between Spectra and Tiling sets for a given $\Omega$.

Fuglede's conjecture is known to be false in dimensions $d \geq 3$ \cite {T} \cite {M} \cite {KM1}. However interest in this conjecture is still alive. For $d=2$  it has been proved for convex planar domains  for $d=2$ \cite {IKT2} and for $d = 3$, Fuglede's conjecture for Convex polytopes in $\R^3$ has been proved very recently by Greenfeld and Lev \cite{LevG} .

 We will restrict to dimension $d=1$. In this case Fuglede's conjecture is known to hold when $\Omega $ is the union of two intervals \cite {L1} and for the case of three intervals the authors proved that Tiling implies Spectral, and except for one situation, Spectral implies tiling too, \cite {BCKM}, \cite {BM2} . In Lagarias and Wang  \cite {LW1} studied the structure of tiling sets $\mathcal T$, and proved that if $(\Omega, \mathcal T)$ is a tiling pair for some $\Omega$, then the tiling set $\mathcal T$ is periodic with an integer period and is rational, i.e. $\mathcal T$ is of the form
$$\mathcal T = \cup_{j=0}^{n-1} (t_j + n\Z)$$
with $t_0 = 0$ and $t_j \in \Q, \, j= 0. 1, ..., n-1$.

Our aim is to study the structure of spectra  $\Lambda$ for spectral pairs $ (\Omega, \Lambda)$. In \cite {BM1}, the authors proved that if  $\Omega$ is the union of finitely many intervals and $ (\Omega, \Lambda)$ is a spectral pair, then $\Lambda$ is periodic with an integer period. In \cite {IK} this result was then proved for any compact set $\Omega$. 

Let $(\Omega, \Lambda)$ be a spectral pair. Since any translate of $\Lambda$ is again a spectrum for $\Omega$, we may assume that $\Lambda$ is of the form
$$\Lambda = \cup_{j=0}^{d-1} (\lambda_j + d\Z) = \Gamma + d\Z,$$
with $\lambda_0=0.$ 
Further, by the structure theorem proved in \cite {BM1} we know that $\Lambda$ is also  a spectrum for a set $\Omega_1$ which is a union
of $d$ equal intervals, whose end points lie on the lattice $\Z/d$;
 i.e.,
$$\Omega_1 = \cup_0^{d-1} [a_j/d, {a_j + 1}/d) $$
$ j = 0,1,...,d-1$, with $a_0 =0$.
 Thus to resolve  the question of rationality of a spectrum, it is enough to assume that the associated spectral  set $\Omega $ is of the above form (such sets are called {\it clusters}). After rescaling, we write
$$\Omega = A + [0,1]$$
with $ A \subset \Z_+, \, 0 \in A,\,\,|A| =d$. Then $\Lambda = \cup_{j=1}^{d-1} (\lambda_j/d + Z) = \Gamma + Z$, and $E_\Lambda  = \{ \frac{1}{\sqrt{d}} e_\lambda(x), \,\, \lambda \in \Lambda\}$ is a complete orthonormal set.

All known spectra of sets in $\R$ are rational;
however it is not known whether this must always be so. In \cite {DL} it is shown that if Fuglede's conjecture is true  in one dimension, then every spectral set of Lebesgue measure $1$ has a rational spectrum.

The following result due to Laba \cite{L2} is the only result we are aware of in the literature, which  addressed the problem of rationality of spectra for clusters:

\begin{theorem}[Laba] Suppose that $\Omega = A + [0,1],\,\, A\subset \Z_+,\,\, |A| = d$ is a spectral set. If $\Omega \subset [0, M]$, where $M  < \frac{5d}{2}$, then all spectra associated to $\Omega$ are rational.
\end{theorem}

In section 2, we deduce two interesting facts about a spectrum from known results. First we observe that elements of any spectrum are either rational of transcendental; next, we relate the rationality of the specrtum to integer zeros of exponenial polynomials. In  Section 3, we show that if for some $\Omega$ such that $(\Omega,\Lambda)$ is a spectral pair and $\Omega$ contains some patterns, then $\Lambda$ has to be rational. In section 4, we prove that if the set $\Omega - \Omega$ contains some rigid structures, then the spectrum is rational.  

\section{}

\subsection{}
We first observe  that the elements of the spectrum are either rational or  transcendental. We explain this below:

We have
$$\widehat{\chi_{\Omega}}(\xi) = (1 - e^{2\pi i \xi})(1+
e^{2\pi i a_1 \xi} + ... + e^{2\pi ia_{d-1} \xi}),$$
 and for every $\lambda_j, \,\, j = 0,1, ..., d-1, \,\,\widehat{\chi_{\Omega}}(\lambda_j ) = 0$. So each $e^{2\pi i\lambda}$ is an algebraic number, in fact, an algebraic integer. Recall the following famous result:

\begin{theorem}[Gelfond-Schneider]
If $\alpha$ and $\beta$ are algebraic numbers with $\alpha \neq 0,
1,$ and if $\beta$ is not a rational number, then any value of
$\alpha^\beta = exp(\beta \log\alpha)$ is a transcendental number.
\end{theorem}

We apply this theorem to  $\alpha = e^{\pi i} = -1$, and $\beta = 2\lambda_k$. Since $\alpha^\beta = e^{2\pi i\lambda_k}$ is
an algebraic integer, $2\lambda_k$ is either rational or is not an
algebraic number. In other words, elements $\lambda_k$ of the spectrum are either rational or transcendental numbers.

\subsection{}{\bf Zeros of Exponential Polynomials.} Consider the tempered distribution obtained as Dirac masses on points of $\Lambda$, i.e  the distribution
$$\delta_\Lambda = \sum_{j=0}^{d-1}\sum_{n \in \Z} \delta_{n +\lambda_j/d} = \delta_\Gamma * \delta_\Z$$
Then
$$ \widehat{\delta_\Gamma}(x) = \sum_{j=0}^{d-1} e^{2\pi i \lambda_j x/d} $$
Recall that with $\Omega$ and $\Lambda$ as above, where $|\Omega| =d$,  $(\Omega,\, \Lambda)$ is a spectral pair iff  $\frac{1}{d} |\widehat{\chi_\Omega}|^2 * \delta_\Lambda \equiv d$ iff $|\widehat{\chi_\Omega}|^2 * \delta_\Gamma * \delta_{\Z}  \equiv d^2$ iff $ ( \Omega - \Omega ) |_{\Z} \subseteq \Z( \widehat{\delta_\Gamma})\cup\{0\}$,
where $ \Z( \widehat{\delta_\Gamma}) = \left\{k \in \Z:  \widehat{\delta_\Gamma} \left(k\right)=0 \right\} $.

\noindent We write $z_j = e^{2 \pi i \lambda_j/d}$, then $$\widehat{\delta_{\Gamma}(k)} = \sum_{j=0}^{d-1}  z_j^k, \,\, k\in \Z$$

We are thus led to study the integer zeros of exponential polynomials; more specifically the integer zeros of exponential polynomials. An important result in this context is the  Skolem-Mahler-Lech theorem, which says that  the sets of integer zeros  of exponential polynomials are of the form $X\cup F$, where $X$ is a union of finitely many complete arithmetic progressions, and $F$ is a finite set. 

In \cite {LW1} Lagarias and Wang considered the case of the exponential polynomial, $\widehat{\delta_{\Gamma}}$, and gave more precise description of the set $X$, which we need to state.

 With the above rescaling, we have , $\Gamma = \{\lambda_0 =0, \lambda_1/d, ..., \lambda_{d-1}/d\}$. we writ $\gamma_ j = \lambda_j$, and define an equivalence relation on $\Gamma$ as $\gamma_i \sim \gamma_j $ iff $\gamma_i - \gamma_j \in \Q$, and we  partition $\Gamma$ into its rational equivalence classes, 
$$\Gamma  = \bigcup_1^k \Gamma_j^*$$

Then $$\widehat{\delta_{\Gamma}}(\xi)=
\widehat{\delta_{\Gamma_1^*}}(\xi)+\dots+\widehat{\delta_{\Gamma_k^*}}(\xi)$$

where,
$$ \Gamma_j^* =\left\{ \gamma_j , \gamma_j+ \frac{l_{j,2}}{m_j}, \gamma_j+ \frac{l_{k,3}}{m_j}, \dots, \gamma_j+ \frac{l_{j,n_j}}{m_j} \right\}$$

Notice

$$\widehat{\delta_{\Gamma_j^*}}(\xi)= e^{2 \pi i \gamma_j \xi} (1+ e^{2 \pi i \frac{l_{j,2}}{m_j}\xi}+ \dots + e^{2 \pi i \frac{l_{j,n_j}}{m_j}\xi})$$

Hence if $$\widehat{\delta_{\Gamma_j^*}}(\xi_0)=0 \mbox{ then }
\widehat{\delta_{\Gamma_j^*}}(\xi_0+m_j p)=0, \forall p \in \Z$$

Thus the zero set of $\widehat{\delta_{\Gamma_j^*}}$ is
$m_j$ periodic.

Let $M:= LCM \left\{m_1,m_2,\dots,m_k\right\}$. Let $X$ be the common integer zero set of
$\{\widehat{\delta_{\Gamma_j^*}}\}$ i.e.
$$X:= \bigcap_{i=1}^k \Z(\widehat{\delta_{\Gamma_j^*}})$$

Then $X$ is $M$ periodic and $X \subseteq \Z(\widehat{\delta_{\Gamma}})$, and $F:= \Z(\widehat{\delta_{\Gamma}}) \setminus X$ is a finite set.

As a consequence, we prove:

\begin{theorem}
Let $(\Omega, \Lambda)$ be a spectral pair as above. Then $\Lambda$ is rational  if and only if $ (A-A) \cap F = \phi $.
\end{theorem}
\begin{proof}
Clearly if $\Lambda$ is rational, then $F= \phi$. For the converse, we have

$$\chi_\Omega * \chi_{\Omega}.\widehat{\delta_{\Gamma}}. \delta_{\Z} = d^2$$
 Consider the equivalence class $\Gamma_1^*$, for which
$$\chi_\Omega * \chi_{ \Omega}.\widehat{\delta_{\Gamma_1}}. \delta_{\Z} =d  |\Gamma_1| \leq d^2.$$
 But 
$$|\widehat{ \chi_\Omega}|^2 * \delta_{\Gamma_1^*} * \delta_{\Z} (0) =  d^2$$
Hence $|\Gamma_1^*| =d$, so that there is only one rational class.
\end{proof}

\begin{remark}

\noindent \begin{enumerate}
\item For the Skolem-Mahler-Lech theorem, the structure and cardinality of the finite set $F$ have been studied (see \cite {E}, and references therein), but "effective" results are largely unknown \cite {E}, \cite {Tblog}.

\item In the case when each of the equivalence classes $\Gamma_j^*$ are singletons, i.e. $\lambda_j - \lambda_k \notin \Q $ for all $j \neq k$, then it can be easily seen that in fact $\Z(\widehat{\delta_\Gamma} ) = F$. (See Corollary 1.20 in \cite {E}).
\end{enumerate}

\end{remark}

\section{}
In this section we show that the existence of some specific structures  or patterns (which we call {\it flags}) in the zero set $\Z(\widehat{\delta_\Gamma})$ guarantees the rationality of $\Gamma$.

We begin with a result which follows from a result of  Jager \cite {Jag}. We state it in our setting and give a proof for the sake completeness. ( Jager's paper  is difficult to find).

\begin{definition}
For  fixed integers $m$, $r$ and $N > r$,   let $S_0 = \{m+1, m+2, ...m+r\}$ and for an $N \geq r$  let $S_n = S_{n-1} + N, \,\, n =1,2, ..., s-1$. These $S_n$'s are called strips and the the set $F = \cup_0^{s-1} S_n$ is called an $r\times s\,\, $-flag. We will think of a flag as an array:

\begin{equation*} \begin{array}{cccc}
m+1 & m+2 & \cdots & m+r \\ 
\circ & \circ & \cdots & \circ\\
m+N+1 & m+N+2 & \cdots & m+N+r \\
\circ & \circ & \cdots & \circ\\
\vdots &\vdots & \ddots & \vdots \\
m+(s-1)N+1 & m+(s-1)N+2 & \cdots & m+(s-1)N+r \\
\circ & \circ & \cdots & \circ\\
\end{array} 
\end{equation*}

\end{definition}
Observe that the $S_n$'s are all disjoint, since $N  > r$. One can think of a flag as a rectangular array of $s$ points  $m+1, m+N+1, ...m+(s-1)N+1$ on a vertical pole, and with $s$ horizontal strips $m +nN +1, m+nN +2, ..., m+nN +r, \,\,\,n = 0,1, ...(s-1)$.

Let $\Gamma = \{0, \lambda_1, ..., \lambda_{d-1} \}$. It is easy to see that If $d >1$, then a $d \times 1$ flag cannot be contained in  $\Z( \widehat{\delta_\Gamma})\cup\{0\} $. ( by a simple Vandermonde argument). However, if $S_0$ is a strip of shorter length, we will consider several such strips in a flag configuration as above, and prove:

\begin{theorem}
Fix two integers $m, r$ with $[\frac{d}{2}] \leq r < d$.  Suppose an $r \times d$ flag $F \subset \Z(\hat{\delta_\Gamma})$. Then $\Gamma$ is rational.
\end{theorem} 
\begin{proof} .  Let $z_j = e^{2\pi i \lambda_j}$. The hypothesis implies that for each $k = 1,2, ..., r $, we have the following system of equations:
\begin{eqnarray*}
\sum_{j=0}^{d-1} z_j^{m+k} & = & 0 \\
\sum_{j=0}^{d-1} z_j^{m+k+N} & = & 0 \\
\vdots \\
\sum_{j=0}^{d-1} z_j^{m+k+(d-1)N} & = & 0
\end{eqnarray*}

Equivalently, for every $k = 1,2,..., r$,
\begin{equation*}
\left(\begin{array}{cccc}
1 & 1 & \cdots & 1  \\
1 & z_1^N & \cdots  & z_{d-1}^N\\
\vdots & \vdots & \ddots & \vdots \\
1 & z_1^{(d-1)N} &  \cdots & z_{d-1}^{(d-1)N} \\
\end{array}\right)
\left(\begin{array}{c}
1\\ z_1^{m+k}\\ \vdots \\ z_{d-1}^{m+k}\\
\end{array} \right)=
\left(\begin{array}{c}
0 \\ 0 \\ \vdots \\ 0\\
\end{array} \right)
\end{equation*}

Thus we can conclude that the above Vandermonde matrix is singular. Hence $z_i^N = z_j^N$ for some $i \neq j$. We define an equivalence relation $z_k \sim z_l  \iff z_k^N = z_l^N$, and we write $\rho_j$ as a representative of each equivalence class so obtained, $ j= 0,1, ..., t$. Also let $[\rho_j] = \{z_{j1}, z_{j2},...z_{jl_j}\}$ be the set of elements ($l_j$ in number) in the $j$th equivalence class. We can now extract a subsystem of the above system of equations:
\begin{equation*}
\left(\begin{array}{cccc}
1 & 1 & \cdots & 1  \\
1 & \rho_1^N & \cdots  & \rho_t^N\\
\vdots & \vdots &  \ddots \vdots \\
1 & \rho_1^{(t-1)N} &  \cdots & \rho_t^{(t-1)N} \\
\end{array}\right)
\left(\begin{array}{c}
\sum_1^{l_0} z_{0s}^{m+k}  \\ \sum_1^{l_1} z_{1s}^{m+k}\\ \vdots \\\sum_1^{l_{t-1}} z_{(t-1)s}^{m+k}\\
\end{array} \right)=
\left(\begin{array}{c}
0 \\ 0 \\ \vdots \\ 0\\
\end{array} \right)
\end{equation*}
Now the $\rho_j$'s are all distinct, and hence the Vandermonde matrix on the left is non-singular. Hence, for each $k= 1,2,...,r$ and $j = 0,1, ... ,t$, we have
$$\sum_{s=1}^{l_j} z_{js}^{m+k} = 0.$$ 
Suppose $t >1$, i.e. there are more than one equivalence classes $[\rho_j]$, then we can choose one equivalence class, say $\rho_{j_0}$ which has less than or equal to $ [\frac{d}{2}] + 1 \leq r$ elements. For this $j_0$, we consider the first $l_{j_0}$ equations from the above set of $r $ equations. We have,
 
\begin{equation*}
\left(\begin{array}{cccc}
z_{j_0 1} & z_{j_0 2} & \cdots & z_{j_0 l_{j_0}}   \\
z_{j_0 1} ^2& z_{j_0 2}^2 & \cdots & z_{j_0 l_{j_0}}^2 \\
\vdots & \vdots &  \ddots \vdots \\
z_{j_0 1}^{l_{j_0}} & z_{j_0 2}^{l_{j_0}} & \cdots & z_{j_0 l_{j_0}}^{l_{j_0}} \\
\end{array}\right)
\left(\begin{array}{c}
 z_{j_0 1}^{m}  \\  z_{j_0 2}^{m}\\ \vdots \\ z_{j_0 l_{j_0}}^{m}\\
\end{array} \right)=
\left(\begin{array}{c}
0 \\ 0 \\ \vdots \\ 0\\
\end{array} \right)
\end{equation*}

\noindent This is a contradiction since the $z_{j_0 i}$'s are all distinct. Hence there is only one equivalence class.
\end{proof}

\begin{remark}
If $\Lambda = \Gamma + \Z$ is a spectrum for a set $\Omega = A + [0,1]$, with $|A| =d$, we know that $A-A \subset \Z \widehat{(\delta_\Gamma)}$. Since $|A-A| \leq \frac{d(d-1)}{2}$, an $r\times d$ flag with $[\frac{d}{2}]+1 \leq r$ cannot be contained in $A-A$. In the next theorem we will improve Jager's's result and show that rationality follows from the existence of a 'shorter' flag in the integer zero set $\Z\widehat{(\delta_\Gamma)}$. Hence if the set $A - A$ itself  has enough structure to contain these shorter flags, then we can conclude the rationality of the spectrum.
\end{remark}

We will now extend a result due to Tijdeman \cite {Tij} , which in our notation can be stated as
\begin{theorem}{\cite {Tij}}
Let $r = d-1$ and suppose that a $(d-1) \times 2$ flag is contained in  $\Z\widehat{(\delta_\Gamma)}$, then the extended $(d-1) \times d$ flag is also contained in $\Z\widehat{(\delta_\Gamma)}$.
\end{theorem}

Our extension of this theorem is for smaller values of $r$.

\begin{theorem}
Suppose that an $r \times (d-r+1)$ flag is contained in  $\Z(\hat(\delta_\Gamma)$, then the extended $r \times d$ flag is also contained in $\Z(\hat(\delta_\Gamma)$.
\end{theorem}

Let $z_0 =1, z_1, ..., z_{d-1}$ be distinct complex numbers and $$f_k =  \sum_0^{d-1} z_j^k$$

Then the $f_k$s satisfy a $d$-term recurrence relation given by the Newton-Girard Formulae:
\begin{equation}
\left(\begin{array}{ccccc}
0 & 1 &  0 &\cdots & 0  \\
 0 &  0 & 1  & \cdots  & 0\\
\vdots & \vdots & \vdots  & \cdots & \vdots \\
0 & 0 & 0 & \cdots & 1\\
c_0 & c_1 &  c_2 & \cdots & c_{d-1}\\
\end{array}\right)
\left(\begin{array}{c}
f_k \\ f_{k+1} \\ \vdots \\  \\ f_{d+k-1}\\
\end{array} \right)=
\left(\begin{array}{c}
f_{k+1} \\ f_{k+2} \\ \vdots  \\  \\ f_{k+ d +1}\\
\end{array} \right)
\end{equation}

In brief, we will write these equations as:
$$U\nu_k = \nu_{k+1}$$

\begin{proof}[Proof of Theorem 6]

Let $F$ be the $r \times (d-r+1) \,\,$ flag contained in $A-A \setminus \{0\} \subset \Z(\hat(\delta_\Gamma)$. Then, each of the vectors $\nu_{m+1},\, \nu_{m+N+1}, ..., \nu_{m+(d-r)N+1}$ will be of the form given below:

\begin{equation*} \left(\begin{array}{c}
0 \\ 0 \\ \vdots \\ 0 \\ * \\ \vdots \\ * \\
\end{array} \right)
\end{equation*}
where the first $r$ entries are $0$. For ease of notation we will write $\mu_j = \nu_{m+jN+1}, \, j= 0,1,...(d-r) $. Clearly these $d-r+1$ vectors are linearly dependent, so there exist constants $\alpha_0, \alpha_1, ...,\alpha_{d-r}$ such that
$$ \mu_{d-r+1} = \alpha_0,\mu_0 + \alpha_1\mu_1+ ...+ \alpha_{d-r} \mu_{d-r}$$

Now we apply the Newton-Girard matrix $U$  $N$ times to get $$U^N\mu_j = \mu_{j+1}$$
Repeating this process, we see that all the vectors $ \nu_{m+jN+1} $ are of the same form. But this means that the extended $r \times d$ flag is also contained in $\Z\widehat{(\delta_\Gamma)}$.
\end{proof}

Combined with Theorem 4, we get

\begin{theorem}
Let $\Lambda = \Gamma + \Z$ be a spectrum for a set $\Omega = A + [0,1]$, with $A \subset \Z \,\,\,|A| =d$. Suppose $r \geq [\frac{d}{2}]$, and  an $r \times (d-r+1)$ flag is contained in  $A - A$, then $\Lambda$ is rational.
\end{theorem}

We have found that the existence of some variations of the flag patterns in $A-A$ again imply rationality of the spectrum.

\end{document}